\documentclass[12pt]{article}
\usepackage[utf8]{inputenc}
\usepackage{amsthm}
\usepackage{amsmath}
\usepackage{bbm}
\usepackage{amsfonts}
\usepackage{amssymb}
\usepackage[left=2cm,right=2cm,top=2cm,bottom=2cm]{geometry}

\usepackage{tabto}
\TabPositions{2 cm, 4 cm, 6 cm, 8 cm}

\usepackage{tikz-cd}

\usepackage{tikz}
\usetikzlibrary{arrows, automata, positioning}

\usepackage{tocloft}
\usepackage{appendix}

\usepackage{enumerate}

\usepackage{hyperref}
 
\newtheorem{theorem}{Theorem}[section]
\newtheorem{corollary}[theorem]{Corollary}
\newtheorem{lemma}[theorem]{Lemma}
\newtheorem{claim}[theorem]{Claim}
\newtheorem{proposition}[theorem]{Proposition}

\theoremstyle{definition}
\newtheorem{definition}[theorem]{Definition}

\newtheorem*{definition*}{Definition}

\newtheorem*{lemma*}{Lemma}
\newtheorem*{proposition*}{Proposition}
\newtheorem*{theorem*}{Theorem}
\newtheorem*{corollary*}{Corollary}

\theoremstyle{definition}

\newtheorem{question}[theorem]{Question}

\newcommand{\im}{\operatorname{im}}
\newcommand{\Hom}{\operatorname{Hom}}
\newcommand{\Homeo}{\operatorname{Homeo}}

\newcommand{\HH}{\operatorname{H}}

\newcommand{\CC}{\operatorname{C}}
\newcommand{\ZZ}{\operatorname{Z}}
\newcommand{\BB}{\operatorname{B}}

\begin{document}

\title{Finitely generated simple left orderable groups \\ with vanishing second bounded cohomology}
\author{Francesco Fournier-Facio and Yash Lodha}
\date{\today}
\maketitle

\begin{abstract}
We prove that the finitely generated simple left orderable groups constructed by the second author with Hyde have vanishing second bounded cohomology, both with trivial real and trivial integral coefficients. As a consequence, these are the first examples of finitely generated non-indicable left orderable groups with vanishing second bounded cohomology. This answers Question $8$ from the $2018$ ICM proceedings article of Andr\'{e}s Navas.
\end{abstract}


\section{Introduction}

A group $G$ is said to be \emph{left orderable} if it admits a total order which is invariant under left multiplication.
This notion has a beautiful connection with dynamics of group actions on the line: a countable group $G$ is left orderable if and only if it admits a faithful action by orientation preserving homeomorphisms on the real line. Two closely related algebraic notions are the following. A group is \emph{indicable} if it admits a surjection onto $\mathbf{Z}$, and \emph{locally indicable} if every finitely generated subgroup is indicable. Note that for a finitely generated group, local indicability implies indicability.

It is well known that locally indicable groups are left orderable, yet the converse fails. However, the converse holds in certain situations.
A fundamental theorem in this direction is Witte-Morris's Theorem \cite{WM}: amenable left orderable groups are locally indicable. Equivalently, finitely generated amenable left orderable groups are indicable. 

In his $2018$ ICM proceedings article \cite{questions}, Andr\'{e}s Navas uses this theorem as a starting point for a research program consisting of a list of questions about left orderable groups which are reminiscent of the Tits alternative. The general theme is: Are some of the properties weaker than amenability enough to imply indicability of a finitely generated left orderable group?
In this paper, we answer the following question from Navas's problem list:

\begin{question}[{\cite[Question 8]{questions}}]
\label{q}
Does there exist a finitely-generated, non-indicable, left orderable group $G$ such that $\HH^2_b(G; \mathbf{R}) = 0$?
\end{question}

Recall that $\HH^\bullet_b(G; \mathbf{R})$ denotes the \emph{bounded cohomology} of $G$ with trivial real coefficients. It is defined analogously to standard group cohomology, but taking the topological dual of the simplicial resolution instead of the algebraic dual (see Section \ref{s:BC} for more details). This invariant was introduced by Johnson and Trauber in the context of Banach algebras \cite{Johnson}, and has since then become a fundamental tool in several fields, most notably the geometry of manifolds \cite{Gromov}, stable commutator length \cite{Calegari} and rigidity theory \cite{rigidity}. Moreover, it is of interest in the setting of group actions on $1$-manifolds because of the work of Ghys \cite{Ghys}, where the second bounded cohomology group with integral coefficients is related in a strong way to actions on the circle.

By a result due to Trauber and first proven by Johnson \cite{Johnson}, amenable groups have vanishing bounded cohomology with trivial real coefficients in all positive degrees (see also \cite[Theorem 3.6]{BC}). In particular, $\HH^2_b(G; \mathbf{R}) = 0$ for every amenable group $G$, which relates Question \ref{q} to Witte-Morris's Theorem. \\

There are two key difficulties in attempting to resolve Question \ref{q}. 
The first is the hypothesis of finite generation. Indeed, using the construction of Mather \cite{Mather} it is easy to show that every  countable left orderable group embeds in a countable left orderable perfect group which has vanishing bounded cohomology in every positive degree \cite{binatebc}. But such groups are not finitely generatable.

When restricting to finitely generated groups, a further difficulty in Question \ref{q} lies in the requirement that the finitely generated group must be \emph{non-indicable}, which for finitely generated groups is stronger that \emph{non-locally indicable}. In other words, the witness to the failure of local indicability is required to be $G$ itself, and not some finitely generated subgroup thereof. In \cite{our}, the authors demonstrated the existence of continuum many finitely generated, left orderable, non-locally indicable groups $G$ such that $\HH^2_b(G; \mathbf{R}) = 0$. However, the groups constructed there are all indicable, and hence do not answer the above question.\\

In \cite{HydeLodha} the second-named author and Hyde constructed the first family of finitely generated simple left orderable groups. The construction takes as input a so-called \emph{quasi-periodic labelling} $\rho$ and outputs a finitely generated simple left orderable group $G_\rho$ (see Section \ref{s:grho} for more detail). Note that the finitely generated groups $G_{\rho}$ are non-indicable, being simple. Thus our main theorem below shows that the groups $G_{\rho}$ provide a positive answer to Question \ref{q}:

\begin{theorem}
\label{theorem:main}
Let $\rho$ be a quasi-periodic labelling. Then $\HH^2_b(G_{\rho}; \mathbf{R}) = 0$.
\end{theorem}

We wish to emphasise here that our solution to Navas's question is even more striking owing to the fact that the groups $G_{\rho}$ are simple, which is much stronger than being non-indicable.

Note that in \cite{HLNR} it was already shown that $G_\rho$ has no unbounded quasimorphisms, which implies that a natural map $\HH^2_b(G_{\rho}; \mathbf{R}) \to \HH^2(G_{\rho}; \mathbf{R})$ is injective. Theorem \ref{theorem:main} recovers this fact, and proves something strictly stronger, since there are many groups with this property which do admit nontrivial second bounded cohomology classes, such as Thompson's group $T$ \cite{our}. 

In particular, we recover the fact that the groups $G_\rho$ are \emph{left orderable monsters}, meaning that every faithful action on the real line is globally contracting \cite[Corollary 0.3]{HLNR}. \\

We are also able to deduce:

\begin{corollary}
\label{cor:integral}

Let $\rho$ be a quasi-periodic labelling. Then $\HH^2_b(G_{\rho}; \mathbf{Z}) = 0$.
Therefore, every action of $G_{\rho}$ on the circle has a global fixpoint.
\end{corollary}

The aforementioned result from \cite{HLNR} already implies that every action on the circle which lifts isomorphically to an action on the real line has a global fixpoint. This corollary extends this result to all actions, and in particular implies that \emph{every} action lifts isomorphically to the real line. \\

To prove Theorem \ref{theorem:main}, we will start by applying the techniques developed by the authors in \cite{our} in order to exhibit several subgroups $H \leq G_{\rho}$ satisfying $\HH^2_b(H; \mathbf{R}) = 0$. The promotion of these computations to $G_{\rho}$ is the novelty of our proof. This involves a new construction of elements in the group $G_{\rho}$ and a refinement of some of the structural analysis of $G_{\rho}$ carried out in \cite{HLNR}.
We believe that some of these techniques are widely applicable, and could lead to further vanishing results for the second bounded cohomology of groups of homeomorphisms. \\

Since the first examples emerged in \cite{HydeLodha}, more families of finitely generated left orderable simple groups have been constructed \cite{MBT, HLR}. It is worth pointing out that Theorem \ref{theorem:main} does not hold for all of them. The commutator subgroups $H_n$ of the fast $n$-ring groups $G_n$, for $n \geq 3$ are finitely generated, left orderable and simple \cite{HLR}. By construction they act on the circle without fixing a point, and since they are also perfect, the action produces a non-zero bounded real Euler class $e \in \HH^2_b(H_n; \mathbf{R})$ \cite{matsumoto} (see also \cite[Corollary 10.28]{BC}). We conjecture that the second bounded cohomology vanishes for the groups of piecewise linear homeomorphism of flows \cite{MBT}. \\

After providing a positive answer to Question \ref{q}, we end this introduction by asking a sharper question, where the vanishing of second bounded cohomology is replaced by \emph{bounded acyclicity}, that is, the vanishing of bounded cohomology with trivial real coefficients in all positive degrees. Several groups of geometric origin have vanishing second bounded cohomology but are not boundedly acyclic \cite[Example 3.10]{FLM}.

\begin{question}
\label{q:bac:fg}
Does there exist a finitely generated, boundedly acyclic, left orderable group that is not indicable? In particular, given a quasi-periodic labelling $\rho$, is $G_\rho$ boundedly acyclic?
\end{question}

By Witte-Morris's Theorem, such a group cannot be amenable. The first examples of finitely generated non-amenable boundedly acyclic groups were recently constructed by the first-named author, L\"oh and Moraschini \cite{FLM}. However all of these examples are indicable and we do not know whether they can be left orderable. Moreover, as mentioned above there exist continuum many countable left-orderable groups which are boundedly acyclic and neither indicable nor locally indicable \cite{binatebc}. We believe that further research into Question \ref{q:bac:fg} will lead to a better understanding of the bounded cohomology of finitely generated groups and of left orderable groups. \\

\textbf{Notation:} In this article, we assume $0\in \mathbf{N}$, and all actions will be right actions (unless function notation is used). Accordingly, we will use the conventions $[g, h] = g^{-1}h^{-1}gh$ for commutators, and $g^h = h^{-1}gh$ for conjugacy. \\

\textbf{Acknowledgements:} The first author was supported by an ETH Z\"urich Doc.Mobility Fellowship. The second author was supported by a START-Projekt Y-1411 of the Austrian Science Fund. The authors would like to thank Matt Brin and Matt Zaremsky for useful discussions and comments.

\section{Bounded cohomology and central extensions}
\label{s:BC}

We will work with cohomology and bounded cohomology with trivial real coefficients, and use the definition in terms of the bar resolution. We refer the reader to \cite{Brown} and \cite{BC, monod} for a general and complete treatment of ordinary and bounded cohomology, respectively. \\

For every $n \geq 0$, denote by $\CC^n(G)$ the set of real-valued functions on $G^n$. By convention, $G^0$ is a single point, so $\CC^0(G) \cong \mathbf{R}$ consists only of constant functions. We define differential operators $\delta^\bullet : \CC^\bullet(G) \to \CC^{\bullet+1}(G)$ as follows:
$\delta^0 = 0$
and for $n \geq 1$:
$$\delta^n(f)(g_1, \ldots, g_{n+1}) = f (g_2, \ldots, g_{n+1}) + $$
$$ + \sum\limits_{i = 1}^n (-1)^i f(g_1, \ldots, g_i g_{i+1}, \ldots, g_{n+1}) + (-1)^{n+1} f(g_1, \ldots, g_n).$$

One can check that $\delta^{\bullet+1}\circ \delta^{\bullet} = 0$, so $(\CC^\bullet(G), \delta^\bullet)$ is a cochain complex. We denote by $\ZZ^\bullet(G) := \ker(\delta^\bullet)$ the \emph{cocycles}, and by $\BB^\bullet(G) := \im(\delta^{\bullet-1})$ the \emph{coboundaries}. The quotient $\HH^\bullet(G) := \ZZ^\bullet(G) / \BB^\bullet(G)$ is the \emph{cohomology of $G$ with trivial real coefficients}. We will also call this \emph{ordinary cohomology} to make a clear distinction from the bounded one, which we proceed to define. \\

Restricting to functions $f :G^\bullet \to \mathbf{R}$ which are bounded, meaning that their supremum $\| f \|_\infty$ is finite, leads to a subcomplex $(\CC^\bullet_b(G), \delta^\bullet)$. We denote by $\ZZ^\bullet_b(G)$ the \emph{bounded cocycles}, and by $\BB^\bullet_b(G)$ the \emph{bounded coboundaries}. The vector space $\HH^\bullet_b(G) := \ZZ^\bullet_b(G) / \BB^\bullet_b(G)$ is the \emph{bounded cohomology of $G$} with trivial real coefficients.

One can similarly define cohomology and bounded cohomology with \emph{trivial integral coefficients} $\HH^n_b(G; \mathbf{Z})$: here $\mathbf{R}$ is replaced by $\mathbf{Z}$ with its Euclidean norm, equipped with the trivial $G$-action. This will only appear in Corollary \ref{cor:integral}. In that setting, we will denote $\HH^\bullet_b(G; \mathbf{R})$ instead of simply $\HH^\bullet_b(G)$ to make a clear distinction between real and integral coefficients.

\subsection{Central extensions}

In degree $2$, cohomology is strongly related to central extensions, and this relation can be exploited to study bounded cohomology as well. We refer the reader to \cite[Chapter 4, Section 3]{Brown} and \cite[Chapter 2]{BC} for detailed accounts and proofs. \\

Since we will be working with trivial real coefficients, throughout this paragraph all central extensions will be of the form:
$$1 \to \mathbf{R} \to E \to G \to 1;$$
where $R := \im(\mathbf{R} \to E)$ is contained in the center of $E$. We will refer to both the short exact sequence above and the group $E$ as a central extension. Such an extension \emph{splits} if there exists a homomorphic section $\sigma : G \to E$, i.e. if it is a direct product. Two central extensions $E, E'$ are \emph{equivalent} if there exists a homomorphism $f : E \to E'$ such that the following diagram commutes:
\[\begin{tikzcd}
	&& E \\
	1 & \mathbf{R} && G & 1 \\
	&& {E'}
	\arrow[from=2-1, to=2-2]
	\arrow[from=2-2, to=1-3]
	\arrow[from=1-3, to=2-4]
	\arrow[from=2-4, to=2-5]
	\arrow[from=2-2, to=3-3]
	\arrow[from=3-3, to=2-4]
	\arrow["f"', from=1-3, to=3-3]
\end{tikzcd}\]
Note that by the $5$-lemma $f$ is automatically an isomorphism. A central extension is split if and only if it is equivalent to the external direct product $\mathbf{R} \times G$. \\

We consider pairs of the form $(E,\sigma)$ where $E$ is a central extension and $\sigma : G \to E$ is a set-theoretic section, which is moreover \emph{normalized}, i.e. it satisfies that $\sigma(id_G) =id_E$.
To every such pair one can associate a $2$-cocycle as follows. Observe that the map $$G^2 \to E : (f, g) \mapsto \sigma(f)\sigma(g)\sigma(fg)^{-1}$$ takes values in $R$. Therefore this can be viewed as a map $\omega : G^2
\to \mathbf{R}$ which can be verified to be a cocycle. This is moreover \emph{normalized}, that is, it satisfies $\omega(f, id_G) = \omega(id_G, f) = 0$ for all $f \in G$.

Conversely, let $\omega$ be a normalized cocycle. Define a group $E$ as follows: as a set, $E = \mathbf{R} \times G$. The group law is defined by the following formula:
$$(\lambda, f) \cdot (\mu, g) := (\lambda + \mu + \omega(f, g), fg).$$
Then $E$ is a group, and the set-theoretic inclusions $\mathbf{R} \to E$ and projection $E \to G$ make $E$ into a central extension. Moreover, the set-theoretic inclusion $\sigma : G \to E$ is a normalized section satisfying $\omega(f, g) = \sigma(f)\sigma(g)\sigma(fg)^{-1}$. \\

This correspondence between normalized cocycles and central extensions with preferred normalized sections, descends to the level of cohomology:

\begin{theorem}[{\cite[Section 4, Chapter 3]{Brown}}]
\label{thm:ext}

There is a bijective correspondence between normalized $2$-cocycles $\omega \in \ZZ^2(G)$ and central extensions $E$ with a normalized section $\sigma : G \to E$. This induces a bijective correspondence between cohomology classes $\alpha \in \HH^2(G)$ and equivalence classes $[E]$ of central extensions.
\end{theorem}

\subsection{$2$-boundedly acyclic groups}

Before moving to the interplay between bounded cohomology and central extensions, we introduce the class of groups which will be most relevant for our purposes.

\begin{definition}
Let $n \geq 1$. We say that $G$ is \emph{$n$-boundedly acyclic} if $\HH^i_b(G) = 0$ for all $1 \leq i \leq n$. It is \emph{boundedly acyclic} if it is $n$-boundedly acyclic for every $n$.
\end{definition}

Since it holds for every group that $\HH^1_b(G) = 0$ \cite[Section 2.1]{BC}, a group $G$ is $2$-boundedly acyclic if and only if $\HH^2_b(G) = 0$. Therefore Theorem \ref{theorem:main} may be restated as: $G_\rho$ is $2$-boundedly acyclic. \\

The main example of boundedly acyclic groups is the following:

\begin{theorem}[Johnson, see {\cite[Chapter 3]{BC}}]
\label{thm:amenable}
Amenable groups are boundedly acyclic.
\end{theorem}

The converse does not hold: there exist boundedly acyclic groups that contain free subgroups. The first such example is due to Matsumoto and Morita \cite{MM}: the group of compactly supported homeomorphisms of $\mathbf{R}^n$. \\

Bounded acyclicity behaves well with respect to amenable extensions:

\begin{theorem}[{Monod \cite[8.6]{monod}, see also \cite{coamenable}}]
\label{thm:amenableext}

Let $n \geq 0$, and let
$$1 \to H \to G \to K \to 1$$
be a group extension such that $K$ is amenable. Then the inclusion $H \to G$ induces an injection in bounded cohomology $\HH^n_b(G) \to \HH^n_b(H)$. In particular, if $H$ is $n$-boundedly acyclic, then so is $G$.
\end{theorem}

It is unknown whether bounded acyclicity is stable under directed unions \cite[Section 4.4]{binatebc}. However, in degree $2$ it is:

\begin{proposition}[{\cite[Corollary 4.16]{binatebc}}]
\label{prop:dirun}

A directed union of $2$-boundedly acyclic groups is $2$-boundedly acyclic.
\end{proposition}

Several examples of $2$-boundedly acyclic groups were exhibited in \cite{our}. The most relevant one for our purposes will be the following:

\begin{theorem}[{\cite[Theorem 1.3]{our}}]
\label{thm:plbac}
Let $G$ be a subgroup of $\textup{PL}^+([0, 1])$. Then $G$ is $2$-boundedly acyclic.
\end{theorem}

Here $\textup{PL}^+([0, 1])$ denotes the group of orientation-preserving piecewise linear homeomorphisms of the interval $[0, 1]$. The proof of Theorem \ref{thm:plbac} makes crucial use of the fact that such homemomorphisms have finitely many breakpoints. Therefore this theorem does not apply to the group $G_\rho$, which consists of \emph{countably singular} piecewise linear homeomorphisms.

\subsection{Homogeneous cocycles}

A special feature of second bounded cohomology is that every class admits a unique canonical representative. This was proven by Bouarich in \cite{bouarich}, as a tool to show that an epimorphism induces an embedding in second bounded cohomology.

\begin{definition}
Let $G$ be a group. A bounded $2$-cocycle $\omega \in \ZZ^2_b(G)$ is \emph{homogeneous} if $\omega(g^i, g^j) = 0$ for every $g \in G$ and every $i, j \in \mathbf{Z}$.
\end{definition}

Clearly every homogeneous cocycle is normalized. They behave well with respect to conjugacy:

\begin{lemma}
\label{lem:conj}

Let $G$ be a group, and $\omega \in \ZZ^2_b(G)$ a bounded homogeneous $2$-cocycle. Let $E$ be the corresponding central extension, and $\sigma : G \to E$ the normalized section. Then for all $f, g \in G$, it holds that $\sigma(f^g) = \sigma(f)^{\sigma(g)}$. In particular, for all $f, g, h \in G$:
$$\omega(f^h, g^h) = \omega(f, g).$$
\end{lemma}

\begin{proof}
The first part of the statement is \cite[Lemma 3.4]{our}. For the second part, we compute
$$\omega(f^h, g^h) = \sigma(f^h)\sigma(g^h)\sigma((f^h g^h)^{-1}) = \sigma(f)^{\sigma(h)}\sigma(g)^{\sigma(h)}\sigma((fg)^{-1})^{\sigma(h)} = \omega(f, g)^{\sigma(h)}.$$
This in turn is equal to $\omega(f, g)$, since $\omega(f, g)$ is central.
\end{proof}

Homogeneous cocycles are of key importance for computations in second bounded cohomology because of the following theorem:

\begin{theorem}[Bouarich \cite{bouarich}]
\label{thm:bouarich}
Every second bounded cohomology class admits a unique homogeneous representative.
\end{theorem}

See \cite[Proposition 2.16]{BC} for a detailed proof. This theorem is especially useful in the case in which the class is trivial, since the $0$-cocycle is homogeneous:

\begin{corollary}
\label{cor:bouarich}

Let $G$ be a group and $H \leq G$ a $2$-boundedly acyclic subgroup. Let $\omega \in \ZZ^2_b(G)$ be a bounded homogeneous $2$-cocycle. Let $E$ be the corresponding central extension, and $\sigma : G \to E$ the normalized section. Then $\sigma|_H : H \to E$ is a homomorphism.
\end{corollary}

\begin{proof}
The bounded cocycle $\omega|_H$ is homogeneous and represents the trivial class, so by Theorem \ref{thm:bouarich} it vanishes identically. Since $\omega(g, h) = \sigma(g)\sigma(h)\sigma(gh)^{-1}$, we conclude.
\end{proof}

The starting point for our proof of Theorem \ref{theorem:main} will be to identify such subgroups $H$ inside $G_\rho$.
As a consequence of Corollary \ref{cor:bouarich}, we also obtain the following:

\begin{proposition}
\label{prop:2dirun}

Let $G$ be a group and suppose that every $2$-generated subgroup of $G$ is $2$-boundedly acyclic. Then $G$ is $2$-boundedly acyclic.
\end{proposition}

\begin{proof}
Let $\alpha \in \HH^2_b(G)$, and let $\omega$ be the unique homogeneous representative given by Theorem \ref{thm:bouarich}. This is associated to a central extension $E$ and a section $\sigma : G \to E$ such that $\omega(f, g) = \sigma(f)\sigma(g)\sigma(fg)^{-1}$ for all $f, g \in G$. Since $\langle f, g \rangle$ is $2$-boundedly acyclic, by Corollary \ref{cor:bouarich} we have $\sigma(fg) = \sigma(f)\sigma(g)$. Therefore $\omega \equiv 0$ and so $\alpha = 0$.
\end{proof}

\section{The group $G_{\rho}$}
\label{s:grho}

Given a group action $G<\textup{Homeo}^+(M)$ and a $g\in G$ for a connected $1$-manifold $M$,
we define $$Supp(g)=\{x\in \mathbf{R}\mid x\cdot g\neq x\}.$$
A homeomorphism $f:[0,1]\to [0,1]$ is \emph{compactly supported}
if $\overline{Supp(f)}\subset (0,1)$.
Similarly, a homeomorphism $f:\mathbf{R}\to \mathbf{R}$ is \emph{compactly supported in} $\mathbf{R}$
if $\overline{Supp(f)}$ is contained in a compact interval in $\mathbf{R}$.

A group action $G \leq \Homeo^+(\mathbf{R})$ is \emph{proximal} if for every compact interval $I$ and every open interval $J$ there exists $g \in G$ such that $I \cdot g \subset J$.

\subsection{Thompson's group $F$ and a certain subgroup $H$}

We only recall the features of Thompson's group $F$ that we need, and we refer the reader to 
\cite{CFP, Belk} for comprehensive surveys.
The group $\textup{PL}^+([0,1])$ is the group of orientation preserving piecewise linear homeomorphisms of $[0,1]$.
\begin{definition}\label{F}
The group $F\leq \textup{PL}^+([0,1])$ consists of homeomorphisms that satisfies the following:
\begin{enumerate}
\item  The breakpoints lie in $\mathbf{Z}[\frac{1}{2}]$.
\item The derivatives, whenever they exist, are integer powers of $2$.
\end{enumerate}
\end{definition}
Here \emph{breakpoints} (or singularity points) are points where the derivative does not exist.
It is well known that $F$ is finitely presented and that $F'$ is simple and consists of precisely the set of elements $g\in F$ such that $\overline{Supp(g)}\subset (0,1)$.

The following subgroup of $F$ will play a key role in the definition of $G_{\rho}$.
\begin{definition}\label{H}
$H$ is the subgroup of elements of $F$ whose slopes at $0,1$ coincide.
\end{definition}

\begin{lemma}[{\cite[Lemma 2.4]{HydeLodha}}]
\label{3gen}
$H$ is $3$-generated. $H'$ is simple and consists of precisely the set of elements of $H$ (or $F$) that are compactly supported in $(0,1)$.
In particular, $H'=F'$.
\end{lemma} 

\begin{definition}\label{Hgenset}
We fix a $3$-element generating set $\{\nu_1,\nu_2,\nu_3\}$ for $H$. 
\end{definition}
The choice of these generators will not be important, so we simply fix an arbitrary such set.

\subsection{Quasi-periodic labellings}
The following notions will play a key role in the definition of $G_{\rho}$.
\begin{definition}\label{labelling}
Consider the set $\frac{1}{2}\mathbf{Z}=\{\frac{1}{2}k\mid k\in \mathbf{Z}\}$.
A \emph{labelling} is a map $$\rho:\frac{1}{2}\mathbf{Z}\to \{a,b,a^{-1},b^{-1}\}$$
which satisfies:
\begin{enumerate}
\item $\rho(k)\in\{a,a^{-1}\}$ for each $k\in \mathbf{Z}$.
\item $\rho(k)\in \{b,b^{-1}\}$ for each $k\in \frac{1}{2}\mathbf{Z}\setminus \mathbf{Z}$.
\end{enumerate}
\end{definition}

It is convenient to view $\rho(\frac{1}{2}\mathbf{Z})$ as a bi-infinite word with respect to the usual ordering of the half integers.
A subset $X\subseteq \frac{1}{2}\mathbf{Z}$ is said to be a \emph{block} if it is of the form 
$$\big\{ k,k+\frac{1}{2}, \ldots ,k+\frac{1}{2}n \big\}$$
for some $k\in \frac{1}{2}\mathbf{Z}, n\in \mathbf{N}$.
Each block is endowed with the ordering inherited from $\frac{1}{2}\mathbf{Z}$.
The set of blocks of $\frac{1}{2}\mathbf{Z}$ is denoted as $\mathbf{B}$.
To each labelling $\rho$ and each block $X=\{k,k+\frac{1}{2},...,k+\frac{1}{2}n\}$, we assign a formal word 
$$W_{\rho}(X) = \rho \big( k \big) \rho \big( k+\frac{1}{2} \big) \ldots \rho \big( k+\frac{1}{2}n \big)$$
which is a word in the letters $\{a,b,a^{-1},b^{-1}\}$.
Such a formal word is called a \emph{subword} of the labelling.

Given a word $w_1 \ldots w_n$ in the letters $\{a,b,a^{-1},b^{-1}\}$, the \emph{formal inverse} of the word is 
$w_n^{-1} \ldots w_1^{-1}$. The formal inverse of $W_{\rho}(X)$ is denoted by $W_{\rho}^{-1}(X)$.

\begin{definition}\label{qplabelling}
A labelling $\rho$ is said to be \emph{quasi-periodic} if the following holds:
\begin{enumerate}
\item For each block $X\in \mathbf{B}$, there is an $n\in \mathbf{N}$ such that whenever 
$Y\in \mathbf{B}$ is a block of size at least $n$, then $W_{\rho}(X)$ is a subword of $W_{\rho}(Y)$.
\item For each block $X\in \mathbf{B}$, there is a block $Y\in \mathbf{B}$ such that $W_{\rho}(Y)=W_{\rho}^{-1}(X)$.
\end{enumerate}
\end{definition}

Note that by \emph{subword} in the above we mean a string of consecutive letters in the word.
An explicit construction of quasi-periodic labellings was provided in \cite[Lemma 3.1]{HydeLodha}.

The following follows directly from Item $2$ in Definition \ref{qplabelling}:

\begin{lemma}\label{lem:notperiodic}
A quasi-periodic labelling is not periodic.
\end{lemma}

\subsection{The definition of $G_{\rho}$}

We shall need the following notation.
Let $I,J$ be nonempty compact intervals of equal length.
We denote by $T_{J,I}:J\to I$ the unique orientation preserving isometry,
and by $T_{J,I}^{or}:J\to I$ as the unique orientation reversing isometry.
Given two isometric closed intervals $I,J\subset \mathbf{R}$, and homeomorphisms $f\in \textup{Homeo}^+(I),g\in \textup{Homeo}^+(J)$, we say that:

\begin{enumerate}
\item $f\cong_T g$ if 
$$f=T_{I,J}\circ f\circ T_{J,I}.$$ 
\item $f\cong_{T^{or}} g$ if 
$$f=T_{I,J}^{or}\circ f\circ T_{J,I}^{or}.$$
\end{enumerate}

Let $H=\langle \nu_1,\nu_2,\nu_3\rangle<\textup{Homeo}^+([0,1])$ be the group from Definition \ref{H}. We define the homeomorphisms 
$$\zeta_1,\zeta_2,\zeta_3,\chi_1,\chi_2,\chi_3:\mathbf{R}\to \mathbf{R}$$ 
as follows: for each $i\in \{1,2,3\}$ and $n\in \mathbf{Z}$,
$$\zeta_i\restriction [n,n+1]\cong_{T}\nu_i \qquad \text{ if } \rho \big( n+\frac{1}{2} \big) = b,$$
$$\zeta_i\restriction [n,n+1]\cong_{T^{or}}\nu_i \qquad \text{ if } \rho \big( n+\frac{1}{2} \big) = b^{-1},$$
$$\chi_i\restriction \big[ n-\frac{1}{2},n+\frac{1}{2} \big] \cong_T \nu_i \qquad \text{ if }\rho(n)=a,$$
$$\chi_i\restriction \big[ n-\frac{1}{2},n+\frac{1}{2} \big] \cong_{T^{or}} \nu_i\qquad \text{ if }\rho(n)=a^{-1}.$$

\begin{definition}\label{SimpleGroup}
To each labelling $\rho$, we 
associate the group
$$G_{\rho}:=\langle \zeta_1^{\pm 1},\zeta_2^{\pm 1},\zeta_3^{\pm 1},\chi_1^{\pm 1},\chi_2^{\pm 1},\chi_3^{\pm 1}\rangle<\textup{Homeo}^+(\mathbf{R}).$$
We denote the above generating set of $G_{\rho}$ as 
$$\mathbf{S}_{\rho}:=\{\zeta_1^{\pm 1},\zeta_2^{\pm 1},\zeta_3^{\pm 1},\chi_1^{\pm 1},\chi_2^{\pm 1},\chi_3^{\pm 1}\}.$$
\end{definition}

We also define subgroup 
$$\mathcal{K}:=\langle \zeta_1^{\pm 1},\zeta_2^{\pm 1},\zeta_3^{\pm 1}\rangle$$
of $G_{\rho}$ which is naturally isomorphic to $H$, via the isomorphism
$$\lambda : H \to \mathcal{K} : \nu_i \mapsto \zeta_i , i = 1, 2, 3.$$
Note that the definition of $\mathcal{K}$ requires us to fix a labelling $\rho$ but we denote them as such for simplicity of notation. \\

The group $G_{\rho}$ is defined for every labelling $\rho$. In the special case in which $\rho$ is quasi-periodic, it is moreover simple:

\begin{theorem}[{\cite[Theorem 1.3]{HydeLodha}}]
Let $\rho$ be a quasi-periodic labelling.
Then the group $G_{\rho}$ is simple.
\end{theorem}

We will also need the following facts about the group $G_\rho$.

\begin{proposition}[{\cite[Lemmas 5.1, 5.3]{HydeLodha}, \cite[Corollary 0.3]{HLNR}}]
\label{prop:dynamics}
Let $\rho$ be a quasi-periodic labelling. Then the following hold:
\begin{enumerate}
    \item Every element in $G_\rho$ fixes a point in $\mathbf{R}$.
    \item The action of $G_\rho$ on $\mathbf{R}$ is minimal.
    \item The action of $G_\rho$ on $\mathbf{R}$ is proximal.
\end{enumerate}
\end{proposition}

\section{The structure of $G_\rho$}

The goal of this section is to recall the notion of \emph{special elements} from \cite{HLNR}, and prove Proposition \ref{Fprime}, which strengthens a key structural result about $G_{\rho}$ from \cite{HLNR}. We fix a quasi-periodic labelling $\rho$ for the rest of this section.

\subsection{Special elements}

For each $n\in \mathbf{Z}$, we denote by $\iota_n$ as the unique orientation reversing isometry $\iota_n:[n,n+1)\to (n,n+1]$. We extend this to a map $\iota:\mathbf{R}\to \mathbf{R}$ as $x\cdot \iota=x\cdot \iota_n$ where $n\in \mathbf{Z}$ is so that $x\in [n,n+1)$. (Note that $\iota_n=T_{[n,n+1),(n,n+1]}^{or}$ as defined before, however this more specialised notation simplifies what appears below.)

Given an $x\in \mathbf{R}$ and $k\in \mathbf{N}$, we define a word $\mathcal{W}(x,k)$ as follows. 
Let $y\in \frac{1}{2}\mathbf{Z}\setminus \mathbf{Z}$ such that $x\in [y-\frac{1}{2},y+\frac{1}{2})$.
Then we define 
$$\mathcal{W}(x,k) = \rho ( y-\frac{1}{2}k ) \rho \big( y-\frac{1}{2}(k-1) \big) \ldots \rho \big( y \big) \ldots 
\rho \big( y+\frac{1}{2}(k-1) \big) \rho \big( y+\frac{1}{2} k \big).$$
We denote by $\mathcal{W}^{-1}(x,k)$ the formal inverse of the word $\mathcal{W}(x,k)$.
Given a compact interval $J\subset \mathbf{R}$ with endpoints in $\frac{1}{2}\mathbf{Z}$ and 
$n \in \mathbf{N}$, we define a word $\mathcal{W}(J,k)$ as follows. Let
$$y_1=inf(J)+\frac{1}{2}, \qquad y_2=sup(J)-\frac{1}{2}.$$
Then we define 
\begin{small}
$$\mathcal{W}(J,k) = \rho \big( y_1-\frac{1}{2}k \big) \rho(y_1-\frac{1}{2}(k-1)) \ldots \rho(y_1) 
\ldots \rho(y_2) \ldots \rho(y_2+\frac{1}{2}(k_2-1))\rho(y_2+\frac{1}{2}k_2).$$
\end{small}
We consider the set of pairs $$\Omega=\{(W,k)\mid W\in \{a,b,a^{-1},b^{-1}\}^{<\mathbf{N}}, k \in \mathbf{N}\text{ such that } |W|=2k+1\}.$$

\begin{definition}\label{specialelementsdefinition}
Recall the map $\lambda$ from Definition \ref{SimpleGroup}.
Given an element $f\in F'$ and $\omega=(W,k) \in \Omega$, 
we define the {\em special element} $\lambda_{\omega}(f)\in \textup{Homeo}^+(\mathbf{R})$ as follows: 
For each $n\in \mathbf{Z}$, we let 
$$\lambda_{\omega}(f)\restriction [n,n+1]=
\begin{cases}
\lambda(f)\restriction [n,n+1]\text{ if }  \mathcal{W}([n,n+1],k)=W^{\pm 1}; \\
id\restriction [n,n+1] \text{ otherwise}.
\end{cases}
$$
\end{definition}

More general special elements were defined in \cite{HLNR}, but this subclass is the only one we will need for our purposes.

\begin{proposition}[{\cite[Proposition 3.5]{HLNR}}]
\label{Specialelements}
Let $\rho:\frac{1}{2}\mathbf{Z}\to \{a,a^{-1},b,b^{-1}\}$ be a quasi-periodic labelling.
Let $\omega\in \Omega$ and $f\in F'$. Then $\lambda_{\omega}(f) \in G_{\rho}$.
\end{proposition}

We recall the following \emph{characterization of elements} of $G_{\rho}$.
This provides an alternative, ``global" description of the groups as comprising of elements satisfying dynamical and combinatorial hypotheses,
and is reminiscent of similar descriptions for various generalisations of Thompson's groups. 

\begin{definition}\label{Krho}
Let $K_{\rho}$ be the set
of homeomorphisms $f\in \textup{Homeo}^+(\mathbf{R})$ satisfying the following:
\begin{enumerate}
\item $f$ is a countably singular piecewise linear homeomorphism of $\mathbf{R}$ with a discrete set of singularities, all of which lie in $\mathbf{Z}[\frac{1}{2}]$.
\item $f'(x)$, wherever it exists, is an integer power of $2$.
\item There is a $k_f\in \mathbf{N}$ such that the following holds.
\begin{enumerate}
\item[3.a] Whenever $x,y\in \mathbf{R}$ satisfy that $$x-y\in \mathbf{Z} \text{ and }  \mathcal{W}(x,k_f)=\mathcal{W}(y,k_f),$$ it holds that $$x-x\cdot f = y-y\cdot f.$$
\item[3.b] Whenever $x,y\in \mathbf{R}$ satisfy that $$x-y\in \mathbf{Z}\text{ and } \mathcal{W}(x,k_f)=\mathcal{W}^{-1}(y,k_f),$$ it holds that $$x-x\cdot f=y'\cdot f-y' \qquad \text{ where }y'=y\cdot \iota.$$
\end{enumerate}
\end{enumerate}
\end{definition}

\begin{theorem}[{\cite[Theorem 1.8]{HLNR}}]
\label{characterisation}

Let $\rho$ be a quasi-periodic labelling.
The groups $K_{\rho}$ and $G_{\rho}$ coincide.
\end{theorem}

\subsection{Interval stabilizers}

The goal of this section is to describe subgroups generated by two elements which pointwise fix an interval. We start with the following definition.

\begin{definition}
\label{def:atom}

A homeomorphism $f\in \textup{Homeo}^+(\mathbf{R})$ is said to be \emph{stable} if there exists an $n\in \mathbf{N}$ 
such that the following holds: For any compact interval $I$ of length at least $n$, there is an integer $m\in I$
such that $f$ fixes a neighbourhood of $m$ pointwise.
Given a stable homeomorphism $f\in \textup{Homeo}^+(\mathbf{R})$ and an interval 
$[m_1,m_2]$, the restriction $f\restriction [m_1,m_2]$ is said to be an \emph{atom of $f$}, if $m_1, m_2 \in \mathbf{Z}$ and $f$ fixes an open neighbourhood of $\{m_1, m_2\}$ pointwise. We will also refer to $[m_1, m_2]$ as the atom: the use of the word will be clear from the context.

Given a stable homeomorphism $f$, we may express $\mathbf{R}$ as a union of intervals with integer endpoints $\{I_{\alpha}\}_{\alpha\in P}$ such that $f\restriction I_{\alpha}$ is an atom for each $\alpha\in P$ and different intervals intersect in at most one endpoint. We call this a \emph{cellular decomposition of $\mathbf{R}$ suitable for $f$}.
\end{definition} 

In \cite{HydeLodha} and \cite{HLNR} the term \emph{atom} refers to atoms in the sense of Definition \ref{def:atom}, which are moreover minimal with respect to inclusion. This definition has the advantage that a cellular decomposition of the real line suitable for a stable homeomorphism $f$ is unique. However, in the proof of Proposition \ref{Fprime} we will need to consider cellular decompositions which are suitable for pairs of elements, so this more relaxed notion will be more convenient.

\begin{definition}
\label{def:conjugateatoms}

Two atoms $f\restriction [m_1,m_2]$ and $f\restriction [m_3,m_4]$ are said to be \emph{conjugate} if $m_2-m_1=m_4-m_3$ and the following holds.
$$f\restriction [m_1,m_2]=h^{-1}\circ f \circ h\restriction [m_3,m_4]\qquad h=T_{[m_1,m_2],[m_3,m_4]};$$
and they are \emph{flip-conjugate} if:
$$f\restriction [m_1,m_2]=h^{-1} \circ f\circ h\restriction [m_3,m_4]\qquad h=T_{[m_1,m_2],[m_3,m_4]}^{or}.$$

Let $f$ be a stable homeomorphism and fix a cellular decomposition $\{I_{\alpha}\}_{\alpha\in P}$ of $\mathbf{R}$ suitable for $f$, and $n \in \mathbf{N}$.
We refine the collection of atoms $\{I_{\alpha}\}_{\alpha\in P}$ into a collection of \emph{decorated atoms}: $$\mathcal{T}_n(f, P)=\{(I_{\alpha},n)\mid \alpha \in P\}.$$
We say that two decorated atoms
$(I_{\alpha}, n)$ and $(I_{\beta}, n)$ are \emph{equivalent} if either of the following holds:
\begin{enumerate}
\item $I_{\alpha},I_{\beta}$ are conjugate and $\mathcal{W}(I_{\alpha},n)=\mathcal{W}(I_{\beta},n)$;
\item $I_{\alpha},I_{\beta}$ are flip-conjugate  and $\mathcal{W}(I_{\alpha},n)=\mathcal{W}^{-1}(I_{\beta},n)$.
\end{enumerate}

The cellular decomposition $\{I_{\alpha}\}_{\alpha\in P}$ of $f$ is said to be \emph{uniform} if there are finitely many equivalence classes of decorated atoms in $\mathcal{T}_n(f, P)$. Note that this definition does not depend on $n$: if there are finitely many equivalence classes in $\mathcal{T}_n(f, P)$, then this is true for any $n\in \mathbf{N}$, since there are finitely many words in $\{a,b,a^{-1},b^{-1}\}^n$.

Let $\zeta$ be an equivalence class of elements in $\mathcal{T}_n(f, P)$.
We define the homeomorphism $f_{\zeta}$ as
$$f_{\zeta}\restriction I_{\alpha}=
\begin{cases}
f\restriction I_{\alpha}\text{ if }(I_{\alpha},n)\in \zeta; \\
f_{\zeta}\restriction I_{\alpha}=id\restriction I_{\alpha}\text{ if }(I_{\alpha}, n)\notin \zeta.
\end{cases}$$
If $\zeta_1, \ldots ,\zeta_m$ are the equivalence classes of elements in $\mathcal{T}_n (f, P)$,
then the list of homeomorphisms $f_{\zeta_1}, \ldots ,f_{\zeta_m}$ is called the \emph{cellular decomposition of $f$} determined by the cellular decomposition $\{ I_\alpha \}_{\alpha \in P}$ of $\mathbf{R}$ and the integer $n \in \mathbf{N}$. Note that $f = f_{\zeta_1} \cdots f_{\zeta_m}$.
\end{definition}

The main result of this section is the following:

\begin{proposition}\label{Fprime}
Let $\rho$ be a quasi-periodic labelling.
Let $f,g\in G_{\rho}$ be elements such that there exists an open interval $I$ which is pointwise fixed by both $f$ and $g$.
Then there is a subgroup $A<G_{\rho}$ isomorphic to a finite direct sum of copies of $F'$, which contains both $f$ and $g$.
\end{proposition}

The statement for a single element is contained in the proof of \cite[Proposition 1.9]{HLNR}. We will prove Proposition \ref{Fprime} without appealing to this.

\begin{proof}
By Item $2$ of Proposition \ref{prop:dynamics}, the action of $G_\rho$ is minimal. Therefore there exists an element $h\in G_{\rho}$ such that $f^h$ and $g^h$ pointwise fix a neighborhood of $0$. It suffices to show the statement for $f^h, g^h$, so we may assume without loss of generality that $I$ is a neighbourhood of $0$.

Let $k := \max\{k_f, k_g\} + 1$ as in Definition \ref{characterisation}. Then for every $x, y \in \mathbf{R}$ such that $x - y \in \mathbf{Z}$ and $\mathcal{W}(x, k - 1) = \mathcal{W}(y, k - 1)$ it holds that
$$x - x \cdot f = y - y \cdot f, \quad x - x \cdot g = y - y \cdot g.$$
If instead $\mathcal{W}(x, k-1) = \mathcal{W}^{-1}(y, k-1)$, then
$$x - x \cdot f = y' \cdot f - y', \quad x - x \cdot g = y' \cdot g - y' \quad \text{where } y' = y \cdot \iota.$$
In particular, if $n \in \mathbf{Z}$ satisfies $\mathcal{W}(n, k) = \mathcal{W}(0, k)$, then both $f$ and $g$ fix pointwise a neighbourhood of $n$ (we took $k$ instead of $k - 1$ in order to ensure that this holds also for a left neighbourhood of $n$). We set $\mathcal{N}$ to be the set of such $n \in \mathbf{Z}$, and note that by the definition of quasi-periodic labelling, $\mathcal{N}$ is infinite and the set of distances between two consecutive elements of $\mathcal{N}$ is bounded.

This determines a decomposition of $\mathbf{R}$ into intervals $\{ I_\alpha \}_{\alpha \in P}$ with endpoints in $\mathcal{N}$. Since $f$ and $g$ fix a neighbourhood of each element of $\mathcal{N}$, each $I_\alpha$ is an atom for both $f$ and $g$, and so $\{ I_\alpha \}_{\alpha \in P}$ is a cellular decomposition of $\mathbf{R}$ which is suitable for both $f$ and $g$. Moreover, since two consecutive elements of $\mathcal{N}$ are at a bounded distance from each other, it follows that $\{ |I_\alpha| \}_{\alpha \in P}$ is bounded. In particular the integer $l := k + \max \{ |I_\alpha| : \alpha \in P \}$ is bounded, and the cellular decomposition $\{ I_\alpha \}_{\alpha \in P}$ is uniform: there are finitely many equivalence classes of decorated atoms in $\mathcal{T}_l(f, P) = \mathcal{T}_l(g, P)$ with respect to both elements $f$ and $g$.
We denote by $\zeta_1,\ldots,\zeta_m$ these equivalence classes, and let $f_{\zeta_1},\ldots,f_{\zeta_m}$ and $g_{\zeta_1},\ldots,g_{\zeta_m}$ be the corresponding cellular decompositions of $f$ and $g$. \\

For each $1\leq i\leq m$, set $L_i := |I_{\alpha}|$ where $(I_{\alpha},l)\in \zeta_i$: this is well-defined since $|I_{\alpha}|=|I_{\beta}|$ whenever $(I_{\alpha},l)$ and $(I_{\beta},l)$ are in the same equivalence class.
For each $1\leq i\leq m$, define the canonical isomorphism $$\phi_i:F'\to F_{[0,L_i]}'$$
where $F_{[0,L_i]}$ is the standard copy of $F$ supported on the interval $[0,L_i]$.

For each $1\leq i\leq m$, we have $$\{\mathcal{W}(I_{\alpha},l)\mid (I_{\alpha},l)\in \zeta_i\}=\{W_i,W^{-1}_i\}$$
for some words $W_1, \ldots, W_m$.
Define a map
$$\phi: \bigoplus_{1\leq i\leq m} F'\to \textup{Homeo}^+(\mathbf{R})$$
as follows. 
For $\alpha\in P$ and $1 \leq i \leq m$:
$$\phi(f_1,\ldots,f_m) \restriction I_{\alpha} \cong_T\phi_i(f_i) \qquad \text{ if }(I_{\alpha},l) \in \zeta_i\text{ and } \mathcal{W}(I_{\alpha},l)=W_i;$$
$$\phi(f_1,\ldots,f_m) \restriction I_{\alpha} \cong_{T^{or}}\phi_i(f_i) \qquad \text{ if }(I_{\alpha},l) \in \zeta_i\text{ and } \mathcal{W}(I_{\alpha},l)=W_i^{-1}.$$

This is an injective group homomorphism, and moreover the image of each element satisfies Definition \ref{Krho} with the uniform constant $l$. This implies that the image of $\phi$ is a subgroup $A < G_\rho$, which is isomorphic to a direct sum of $m$ copies of $F'$.
Moreover, for every $1 \leq i \leq m$, there is an element $f_i \in F'$ such that
$$\phi(id, \ldots, f_i, \ldots, id) = f_{\zeta_i}.$$
Therefore $f_{\zeta_i} \in A$ for every $i$, and thus $f = f_{\zeta_1} \cdots f_{\zeta_m} \in A$ as well.
Similarly $g_{\zeta_i} \in A$ for every $i$ and thus $g = g_{\zeta_1} \cdots g_{\zeta_m} \in A$.
This concludes the proof.
\end{proof}

\subsection{$2$-boundedly acyclic subgroups of $G_\rho$}

We are ready to identify certain $2$-boundedly acyclic subgroups of $G_\rho$.

\begin{proposition}\label{prop:stabbac}
Let $\Gamma \leq G_\rho$ and suppose that $\Gamma$ admits a global fixpoint $x\in \mathbf{R}$. Then $\Gamma$ is $2$-boundedly acyclic.
\end{proposition}

\begin{proof}
Since $x$ is a global fixpoint of $\Gamma$, and the left and right slope at $x$ are powers of $2$, we obtain a germ homomorphism $\Gamma \to \mathbf{Z}^2$, whose kernel is the subgroup $\Gamma_1$ consisting of elements $f\in \Gamma$ that pointwise fix some open neighbourhood $I_f$ of $x$. By Theorems \ref{thm:amenable} and \ref{thm:amenableext}, it suffices to show that $\Gamma_1$ is $2$-boundedly acyclic.

Every $2$-generated subgroup $\Delta \leq \Gamma_1$ pointwise fixes a nonempty open interval which contains $x$. Therefore by Proposition \ref{Fprime}, every such $\Delta$ is isomorphic to a subgroup of a finite direct sum of copies of $F'$. In particular, $\Delta$ is isomorphic to a subgroup of $\textup{PL}^+([0, 1])$ and so it is $2$-boundedly acyclic by Theorem \ref{thm:plbac}. Thus $\Gamma_1$ satisfies the hypotheses of Proposition \ref{prop:2dirun}, and so it is $2$-boundedly acyclic.
\end{proof}

\section{Proof of Theorem \ref{theorem:main}}

The following proposition is the key idea behind the proof of the main Theorem.

\begin{proposition}
\label{prop:main}
Let $\rho$ be a quasi-periodic labelling. Let $a_1, a_2, a_3 \in G_\rho$. Then there exist $h, g_1, g_2, g_3 \in G_\rho$ and non-empty sets $I_1 \subset I_2 \subset I_3 \subset \mathbf{R}$ with the following properties. Let $f_i := a_i^h$ for $i = 1, 2, 3$.
\begin{enumerate}
    \item Each of the groups $\langle g_1, g_2, g_3 \rangle$ and $\langle f_i, g_i \rangle : i = 1, 2, 3$ has a global fixpoint in $\mathbf{R}$.
    \item The element $f_i g_i$ fixes $I_i$ pointwise, for $i = 1, 2, 3$.
    \item $Supp(g_i) \subset I_{i+1}$ for $i = 1, 2$.
\end{enumerate}
\end{proposition}

Let us see how this proposition implies Theorem \ref{theorem:main}:

\begin{proof}[Proof of Theorem \ref{theorem:main}]
Let $\alpha \in \HH^2_b(G_\rho)$, and let $\omega$ be the unique homogeneous representative given by Theorem \ref{thm:bouarich}. This is associated to a central extension $E$ and a section $\sigma : G_\rho \to E$ such that $\omega(f, g) = \sigma(f)\sigma(g)\sigma(fg)^{-1}$. By Corollary \ref{cor:bouarich}, for every $2$-boundedly acyclic subgroup $\Gamma \leq G_\rho$, the restriction $\sigma|_H$ is a homomorphism. In particular, this holds when $\Gamma$ is abelian, by Theorem \ref{thm:amenable}, or has a global fixpoint, by Proposition \ref{prop:stabbac}.

We need to show that $\omega \equiv 0$, equivalently that $\sigma$ is a homomorphism. This means that for all $a, b \in G_\rho$ it holds $\sigma(a)\sigma(b)\sigma((ab)^{-1}) = \sigma(a)\sigma(b)\sigma(ab)^{-1} = id$. (Here we used that $\sigma(g^{-1}) = \sigma(g)^{-1}$, since $\sigma$ is a homomorphism when restricted to a cyclic subgroup.) So in other words we need to show that whenever $a_1, a_2, a_3 \in G_\rho$ satisfy $a_1 a_2 a_3 = id$, we also have $\sigma(a_1)\sigma(a_2)\sigma(a_3) = id$. \\

So let $a_1, a_2, a_3$ be as above, and let $h, g_1, g_2, g_3$ and $I_1, I_2, I_3$ be given by Proposition \ref{prop:main}. We similarly set $f_i = a_i^h$ for $i = 1, 2, 3$, and notice that $f_1 f_2 f_3 = id$, and that by Lemma \ref{lem:conj} it suffices to show that $\sigma(f_1)\sigma(f_2)\sigma(f_3) = id$.

\begin{claim}
It holds
$$\sigma(f_1)\sigma(f_2)\sigma(f_3) \cdot \sigma(g_3) \sigma(g_2) \sigma(g_1) = \sigma(f_1 g_1) \sigma(f_2 g_2) \sigma(f_3 g_3).$$
\end{claim}

\begin{proof}[Proof of Claim]
By Proposition \ref{prop:stabbac}, the group $\langle f_i, g_i \rangle$ is $2$-boundedly acyclic. Therefore:
$$\sigma(f_1)\sigma(f_2) \cdot \sigma(f_3) \sigma(g_3) \cdot \sigma(g_2) \sigma(g_1) = \sigma(f_1)\sigma(f_2) \cdot \sigma(f_3 g_3) \cdot \sigma(g_2) \sigma(g_1).$$
Now $f_3 g_3$ fixes $I_3$ pointwise, and $Supp(g_2)\subset I_3$, therefore $\langle f_3 g_3, g_2 \rangle$ is abelian and:
$$\sigma(f_1)\sigma(f_2) \cdot \sigma(f_3 g_3) \sigma(g_2) \cdot \sigma(g_1) = \sigma(f_1)\sigma(f_2) \cdot \sigma(g_2) \sigma(f_3 g_3) \cdot \sigma(g_1).$$
Similarly, $Supp(g_1)\subset I_2 \subset I_3$, therefore $\langle f_3 g_3, g_1 \rangle$ and $\langle f_2 g_2, g_1 \rangle$ are abelian. Thus:
$$\sigma(f_1)\sigma(f_2) \sigma(g_2) \cdot \sigma(f_3 g_3) \sigma(g_1) = \sigma(f_1)\sigma(f_2) \sigma(g_2) \cdot \sigma(g_1) \sigma(f_3 g_3) =$$
$$=\sigma(f_1) \cdot \sigma(f_2g_2) \sigma(g_1) \cdot \sigma(f_3 g_3) = \sigma(f_1) \cdot \sigma(g_1) \sigma(f_2 g_2) \cdot \sigma(f_3g_3) = \sigma(f_1 g_1)\sigma(f_2 g_2)\sigma(f_3 g_3),$$
which concludes the proof of the claim.
\end{proof}

Now the group $\langle f_1 g_1, f_2 g_2, f_3 g_3 \rangle$ fixes $I_1$ pointwise, which is a non-empty set. By Proposition \ref{prop:stabbac}, it is $2$-boundedly acyclic. Therefore
$$\sigma(f_1 g_1) \sigma(f_2 g_2) \sigma(f_3 g_3) = \sigma(f_1 g_1 \cdot f_2 g_2 \cdot f_3 g_3).$$
Using the same argument as in the proof of the claim in reverse, we obtain
$$f_1 g_1 \cdot f_2 g_2 \cdot f_3 g_3 = f_1 f_2 f_3 \cdot g_3 g_2 g_1 =  g_3 g_2 g_1.$$
Therefore
$$\sigma(f_1 g_1 \cdot f_2 g_2 \cdot f_3 g_3) = \sigma(g_3 g_2 g_1) = \sigma(g_3) \cdot \sigma(g_2) \cdot \sigma(g_1),$$
where we used that $\langle g_1, g_2, g_3 \rangle$ is $2$-boundedly acyclic, by Proposition \ref{prop:stabbac}.

We have thus shown that
$$\sigma(f_1)\sigma(f_2)\sigma(f_3) \cdot \sigma(g_3) \sigma(g_2) \sigma(g_1) = \sigma(g_3) \sigma(g_2) \sigma(g_1).$$
This implies that $\sigma(f_1)\sigma(f_2)\sigma(f_3) = id$, which is what we wanted to show.
\end{proof}
Corollary \ref{cor:integral} is then just a combination of Theorem \ref{theorem:main} with previously established facts.

\begin{proof}[Proof of Corollary \ref{cor:integral}]
The inclusion $\mathbf{Z} \subset \mathbf{R}$ induces a long exact sequence \cite[Proposition 1.1]{gersten} (see also \cite[Proposition 8.2.12]{monod}) that begins as follows:
$$0 \to \HH^1(G_\rho; \mathbf{R}/\mathbf{Z}) \to \HH^2_b(G_\rho; \mathbf{Z}) \to \HH^2_b(G_\rho; \mathbf{R}) \to \cdots$$
Since $G_\rho$ is perfect, we have $\HH^1(G_\rho; \mathbf{R}/\mathbf{Z}) \cong \Hom(G_\rho; \mathbf{R}/\mathbf{Z}) = 0$. Therefore Theorem \ref{theorem:main} implies that $\HH^2_b(G_\rho; \mathbf{Z}) = 0$.

The statement about circle actions is a direct consequence of Ghys's Theorem \cite{Ghys} (see also \cite[Proposition 10.20]{BC}).
\end{proof}
The rest of this section will be devoted to the proof of Proposition \ref{prop:main}.

\subsection{Proof of Proposition \ref{prop:main}}

Let $a_1, a_2, a_3 \in G_\rho$. By Item $1$ of Proposition \ref{prop:dynamics}, each $a_i$ fixes a point in $\mathbf{R}$. Therefore there exists a compact interval $J_1 \subset \mathbf{R}$ with endpoints in $\mathbf{Z}[\frac{1}{2}]$ such that each $a_i$ fixes a point inside $J_1$. Next, choose compact intervals $J_2, J_3, J_4$ with endpoints in $\mathbf{Z}[\frac{1}{2}]$ such that
$$J_i \cup J_i \cdot a_i \subset J_{i+1} : 1 \leq i \leq 3.$$

By Item $3$ of Proposition \ref{prop:dynamics} and the definition of a proximal action, there exists an element $h \in G_\rho$ such that $J_4 \cdot h \subset (0, 1)$. We set $T_i := J_i \cdot h$, and note that $T_i \subset (0, 1)$ has endpoints in $\mathbf{Z}[\frac{1}{2}]$, for $i = 1, \ldots, 4$, since $G_\rho$ preserves $\mathbf{Z}[\frac{1}{2}]$.

As in the statement of Proposition \ref{prop:main}, we set $f_i := a_i^h$. By Definition \ref{Krho} and Theorem \ref{characterisation}, there exists $k \in \mathbf{N}$ such that whenever $x\in [0,1], y \in [n,n+1]$ satisfy $x - y\in \mathbf{Z}$ and $\mathcal{W}([0, 1], k) = \mathcal{W}([n, n+1], k)$, it holds that:
$$x - x \cdot f_i = y - y \cdot f_i : i = 1, 2, 3.$$
Fix $m \neq 0$ such that $\mathcal{W}([0,1], k) = \mathcal{W}([m,m+1], k)$: such an $m$ must exist because $\rho$ is quasi-periodic. However $\rho$ is not periodic by Lemma \ref{lem:notperiodic}, therefore there exists $l > k$ such that $\mathcal{W}([0,1], l) \neq \mathcal{W}([m, m+1], l)$. Moreover, since $\mathcal{W}([0,1], k) = \mathcal{W}([m,m+1], k)$, we also have $\mathcal{W}([0,1], l) \neq \mathcal{W}^{-1}([m,m+1], l)$.

We set $W := \mathcal{W}([0,1], l), \omega := (W, l) \in \Omega$ and $$\mathcal{N}_1 := \{ n \in \mathbf{Z} : \mathcal{W}([n,n+1], l) = W \}\qquad \mathcal{N}_2 := \{ n \in \mathbf{Z} : \mathcal{W}([n,n+1], l) = W^{-1} \}.$$
Note that by construction $m \notin \mathcal{N}_1\cup \mathcal{N}_2$. Finally, we set $$I_i := (T_i + \mathcal{N}_1)\bigcup (T_i\cdot \iota+\mathcal{N}_2).$$ for $i = 1, \ldots, 4$. Then these sets satisfy $\emptyset \subset I_1 \subset I_2 \subset I_3 \subset I_4 \subset \mathbf{R}$, and all of these inclusions are strict. \\

We will now use $\omega$ to define the elements $g_i$: these will be special elements as in Definition \ref{Specialelements}. For $i = 1, 2, 3$, let $\alpha_i$ be an element of $F$ supported on $T_{i+1} \subset (0, 1)$ such that $\alpha_i \restriction T_i\cdot f_i = f_i^{-1} \restriction T_i \cdot f_i$. This is possible since $T_i \cdot f_i \subset T_{i+1} \subset (0, 1)$ for all $i = 1, 2, 3$, and the restriction $f_i^{-1} \restriction T_i \cdot f_i$ is piecewise linear with breakpoints in $\mathbf{Z}[\frac{1}{2}]$ and slopes (when they exist) equal to powers of $2$. Note that $\alpha_i \in F'$, and so we can define $g_i := \lambda_\omega(\alpha_i)$. \\

Now that we have defined all objects involved, we will prove that each item of Proposition \ref{prop:main} holds.

\begin{claim}
\label{claim:support}

For $i = 1, 2, 3$, we have $Supp(g_i) \subset I_{i+1}$. In particular, $\langle g_1, g_2, g_3 \rangle$ has a global fixpoint.
\end{claim}

\begin{proof}
By construction, $\alpha_i \in F$ is supported on $T_{i+1} \subset (0, 1)$. Let $x \in \mathbf{R}$ be such that $x \cdot g_i \neq x$, and let $n \in \mathbf{Z}$ be such that $x \in [n, n+1]$. Since $g_i = \lambda_\omega(\alpha_i)$, it holds that either: $\mathcal{W}([n, n+1], l) = W$ or $\mathcal{W}([n, n+1], l) = W^{- 1}$. Then by definition, the following holds. In the former case, $n \in \mathcal{N}_1$ and $x \in T_{i+1} + n \subset I_{i+1}$. In the latter case, $n \in \mathcal{N}_2$ and $x \in T_{i+1}\cdot \iota + n \subset I_{i+1}$. This shows that $Supp(g_i) \subset I_{i+1}$.

Now $\langle g_1, g_2, g_3 \rangle$ is supported on $I_4$, and so every point in $\mathbf{R} \setminus I_4 \neq \emptyset$ is a global fixpoint. (Note that $\mathbf{Z}$ will, in particular, be fixed pointwise.)
\end{proof}

\begin{claim}
\label{claim:fixpoint}

For $i = 1, 2, 3$, the group $\langle f_i, g_i \rangle$ has a global fixpoint, and $f_i g_i$ fixes $I_i$ pointwise.
\end{claim}

\begin{proof}
Fix $i \in \{1, 2, 3\}$. Recall that we fixed an element $m \in \mathbf{Z}$ such that $m \notin \mathcal{N}_1\cup \mathcal{N}_2$ (that is $\mathcal{W}([m,m+1], l) \neq W^{\pm 1}$), but $\mathcal{W}([m, m+1], k) = \mathcal{W}([0, 1], k)$. Moreover, recall that $f_i$ has a fixpoint $x_i$ inside $T_1$, because $a_i$ has a fixpoint inside $J_1$. It follows from the definition of $k$ that $m + x_i \in [m, m+1]$ is a fixpoint of $f$. Moreover, from Claim \ref{claim:support}, it follows that $g_i$ fixes $[m, m+1]$ pointwise. Therefore $m + x_i$ is a global fixpoint of $\langle f_i, g_i \rangle$.

By construction, $$f_i^{-1} \restriction T_i \cdot f_i = \alpha_i \restriction T_i \cdot f_i = g_i \restriction T_i \cdot f_i$$ Therefore, $f_i g_i$ fixes $T_i$ pointwise.

Now let $x\in I_i$. We wish to show that $x\cdot f_ig_i=x$. Let $n\in \mathbf{Z}$ be such that $x\in [n,n+1)$.
Then either $n\in \mathcal{N}_1$ and $x\in T_i+n$, or $n\in \mathcal{N}_2$ and $x\in T_i\cdot \iota+n$.

By definition of $l$, the fact that $l>k$, and construction of the $g_i$, we have the following. If $x$ is a fixpoint of $f_i g_i$ and $y \in \mathbf{R}$ satisfies $x - y = n \in \mathcal{N}_1$, then $y$ is also a fixpoint of $f_i g_i$. Similarly, if $x$ is a fixpoint of $f_i g_i$ and $y \in \mathbf{R}$ satisfies $x - y\cdot \iota = n \in \mathcal{N}_2$, then $y$ is also a fixpoint of $f_i g_i$. Our claim follows.
\end{proof}

\begin{proof}[Proof of Proposition \ref{prop:main}]
Item $1$ follows from Claims \ref{claim:support} and \ref{claim:fixpoint}, Item $2$ from Claim \ref{claim:fixpoint} and Item $3$ from Claim \ref{claim:support}.
\end{proof}

\bibliographystyle{abbrv}
\bibliography{navas}

\noindent{\textsc{Department of Mathematics, ETH Z\"urich, Switzerland}}

\noindent{\textit{E-mail address:} \texttt{francesco.fournier@math.ethz.ch}} \\

\noindent{\textsc{Department of Mathematics, Universit\"at Wien, Austria}}

\noindent{\textit{E-mail address:} \texttt{yash.lodha@univie.ac.at}}

\end{document}